\documentclass{amsart}

\usepackage{amsmath,amsthm,amsfonts,amssymb,amscd,latexsym}

\newtheorem{thm}{Theorem}[section]
\newtheorem{lem}[thm]{Lemma}
\newtheorem{prop}[thm]{Proposition}
\newtheorem{cor}[thm]{Corollary}
\newtheorem{defi}{Definition}
\newtheorem{rem}{Remark}

\numberwithin{equation}{section}

\newcommand{\Hom}{\mathrm{Hom}}

\newcommand{\coker}{\mathrm{Coker}}

\newcommand{\leib}{\mathrm{Leib}}
\newcommand{\lie}{\mathrm{Lie}}

\newcommand{\rel}{\mathrm{rel}}

\newcommand{\N}{\mathbb{N}}
\newcommand{\F}{\mathbb{F}}

\newcommand{\lf}{\mathfrak{L}}

\newcommand{\qf}{\mathfrak{q}}
\newcommand{\gf}{\mathfrak{g}}

\newcommand{\hf}{\mathfrak{h}}

\begin{document}


\title{Spectral sequences for commutative Lie algebras}

\author{Friedrich Wagemann}
\address{Laboratoire de math\'ematiques Jean Leray, UMR 6629 du CNRS, Universit\'e
de Nantes, 2, rue de la Houssini\`ere, F-44322 Nantes Cedex 3, France}
\email{wagemann@math.univ-nantes.fr}




\begin{abstract}
We construct some spectral sequences as tools for computing commutative cohomology of commutative Lie algebras in characteristic $2$. In a first part, we focus on a Hochschild-Serre-type spectral sequence, while in a second part we obtain comparison spectral sequences which mediate between Chevalley-Eilenberg-, commutative- and Leibniz cohomology. These methods are illustrated by a few computations. 
\end{abstract}


\date{August 17, 2019}
          
\maketitle


\section*{Introduction}

The classification of finite-dimensional simple Lie algebras over a field $\F$ of characteristic $2$ is still widely open. It is known from the classification of Lie algebras over a field of characteristic zero that cohomological methods may be powerful allies in the classification quest. In \cite{LZ}, there has been recently introduced a new cohomology theory for Lie algebras in characteristic $2$, called {\it commutative cohomology} \cite{LZ}. It is more generally defined on so-called {\it commutative Lie algebras} in characteristic $2$, i.e. vector spaces $\gf$ over $\F$ with a bracket which satisfies the usual Jacobi identity and on top of that $[x,y]=[y,x]$ for all $x,y\in\gf$. The corresponding cohomology theory is then based on the {\it symmetric} tensor powers of $\gf$ with the usual Chevalley-Eilenberg differential.

It has been illustrated in \cite{LZ} that commutative cocycles and commutative cohomology arise in the classification of finite-dimensional simple Lie algebras in characteristic $2$, see \cite{LZ} and references therein. The authors of \cite{LZ} listed a whole catalogue of questions concerning commutative cohomology, and we will answer some of them in the present paper. We hope that the computational methods of the present paper will serve in future computations of commutative cohomology. 

We focus in the present paper on two cohomological tools for the computation of commutative cohomology. In Sections 1-3, we develop a Hochschild-Serre-type spectral sequence for commutative cohomology. The construction is very close to the original construction by Hochschild and Serre \cite{HS}. In Section 4, we perform some cohomology computations with the help of the spectral sequence, notably for the two dimensional commutative Lie algebra ${\mathfrak N}$ generated by $e$ and $f$ with the only relation $[f,f]=e$ and for the two dimensional Lie algebra ${\mathfrak a}$ generated by $e$ and $h$ with the relations 
$[h,e]=[e,h]=e$. On the other hand, we obtain a cohomology vanishing  theorem for any commutative Lie algebra $\gf$ with a $1$-dimensional ideal with values in the $1$-dimensional non-trivial $\gf$-module $\F_1$, see Theorem \ref{cohomology_vanishing}. 

The second part of the paper in Sections 5-6 concerns three comparison spectral sequences which mediate between Chevalley-Eilenberg- and Leibniz cohomology, between Chevalley-Eilenberg- and commutative cohomology, and between commu\-ta\-tive- and Leibniz cohomology, respectively. In order to illustrate the use of these comparison spectral sequences, we show for example in Theorem  \ref{vanishing_all_cohomologies} that for any Lie algebra $\gf$ with a $1$-dimensional ideal with values in the $1$-dimensional non-trivial $\gf$-module $\F_1$, the vanishing of commutative cohomology implies the vanishing of Leibniz- and Chevalley-Eilenberg cohomology as well.  

Note that there are more spectral sequences which should transpose easily to commutative Lie algebras, as for example the Feigin-Fuchs spectral sequence for $\N$-graded Lie algebras.          


\section{Commutative Lie algebras and their cohomology}

Let $\F$ be a field of characteristic $2$. 

\begin{defi}
A commutative Lie algebra over $\F$ is an $F$-vector space $\gf$ with a bilinear bracket 
$[,]:\gf\times\gf\to\gf$ such that for all $x,y,z\in\gf$
$$[x,y]=[y,x],\,\,\,{\rm and}\,\,\,[x,[y,z]]+[y,[z,x]]+[z,[x,y]]=0.$$
\end{defi}

\begin{rem}
In particular, a commutative Lie algebra is a left and right Leibniz algebra, i.e. a {\it symmetric Leibniz algebra}.  
\end{rem}

As usual, an $\F$-vector space $M$ together with an $\F$-bilinear map $\gf\times M\to M$,
denoted $(x,m)\mapsto x\cdot m$, is called a $\gf$-module in case for all $x,y\in \gf$ and all
$m\in M$, we have 
$$[x,y]\cdot m=x\cdot(y\cdot m)+y\cdot(x\cdot m).$$ 

Following \cite{LZ}, we next define a cochain complex for commutative Lie algebras which will have its applications in the study and classification of Lie algebras in characteristic $2$. 

\begin{defi}
Let $\gf$ be a commutative Lie algebra and $M$ be a $\gf$-module. We set 
$$CS^n(\gf,M):=\Hom_\F(S^n\gf,M),$$
where $S^n\gf$ denotes the usual symmetric algebra on the vector space $\gf$. 
The graded vector space $CS^\bullet(\gf,M)$ becomes a cochain complex with the usual Chevalley-Eilenberg differential
\begin{eqnarray}
d^nf(x_1,\dots,x_{n+1}) & := & \sum_{i=1}^{n+1}x_i\cdot f(x_1,
\dots,\hat{x}_i,\dots,x_{n+1})   \nonumber  \\
& + & \sum_{1\le i<j\le n+1}f([x_i,x_j],x_1,\dots,\hat{x}_i,\dots,\hat{x}_j,\dots,x_{n+1})
\label{CE_coboundary}
\end{eqnarray}
for any $f\in CS^n(\gf,M)$ and all elements $x_1,\dots,x_{n+1}\in\gf$. 
The corresponding cohomology is called commutative cohomology of commutative Lie algebras and denoted by $HS^\bullet(\gf,M)$.  
\end{defi}

\begin{rem}
Observe that this cohomology is different from the usual Lie algebra cohomology in characteristic $2$, where one takes cochain spaces $C^q(\gf,M)$ which vanish in case $q$ exceeds the dimension of $\gf$. 
\end{rem} 

We next show that we have the usual Cartan relation for the symmetric cohomology of commutative Lie algebras. 

\begin{prop}  \label{prop_Cartan_relation}
Let $\gf$ be a commutative Lie algebra over $\F$. Then we have for all $x\in\gf$
$$L_x=d\circ i_x + i_x\circ d,$$
where $L_x$ is the Lie derivative, i.e. for all $f\in CS^n(\gf,M)$ and all $x_1,\ldots,x_n\in\gf$, we have 
$$(L_xf)(x_1,\ldots,x_n):=x\cdot f(x_1,\ldots,x_n)+\sum_{i=1}^nf(x_1,\ldots,[x,x_i],\ldots,x_n),$$
and $i_x$ is the usual insertion operator, i.e. for all $f\in CS^n(\gf,M)$ and all $x_1,\ldots,x_{n-1}\in\gf$, we have 
$$(i_xf)(x_1,\ldots,x_{n-1}):=f(x,x_1,\ldots,x_{n-1}).$$
\end{prop}

\begin{proof}
Let $f\in CS^n(\gf,M)$ and $x,x_1,\ldots,x_n\in\gf$. When computing 
$$(i_x(df))(x_1,\ldots,x_n)=df(x,x_1,\ldots,x_n),$$
there are the terms involving $x$ in the bracket terms or in the action term of the coboundary operator, and there are the terms where $x$ is simply inserted as the first argument of $f$.
These last terms constitute exactly $d(i_xf)(x_1,\ldots,x_n)$, because in $d(i_xf)$, the element $x$ is never involved in the action or in the bracket. The remaining terms constitute 
$(L_xf)(x_1,\ldots,x_n)$. 
\end{proof}


\section{The spectral sequence associated to a subalgebra}

In this section, we construct a Hochschild-Serre-type spectral sequence, closely inspired by 
\cite{HS}, which is designed to compute the commutative cohomology of a commutative Lie algebra $\gf$ which admits a subalgebra $\hf$.  

The filtration leading to the spectral sequence is defined as follows. 
$${\mathcal F}^pCS^n(\gf,M):=\{c\in CS^n(\gf,M)\,|\,c(x_1,\ldots,x_n)=0\,\,{\rm if}\,\,n-p+1\,\,
{\rm elements}\,\,x_i\,\,{\rm belong}\,\,{\rm to}\,\,\hf\}.$$

\begin{lem}
The filtration is compatible with the Chevalley-Eilenberg coboundary operator (\ref{CE_coboundary}), i.e.
$$d{\mathcal F}^pCS^n(\gf,M)\subset{\mathcal F}^pCS^{n+1}(\gf,M).$$
\end{lem} 

\begin{proof}
Consider Equation (\ref{CE_coboundary}) with $n+1-p+1=n-p+2$ elements in $\hf$. We have to show that all terms in this equation are zero in case the cochain $f$ vanishes on $n-p+1$ elements of $\hf$. The action terms vanish when $n-p+2$ elements are in $\hf$, because there always remain at least $n-p+1$ elements in $\hf$ as arguments of $f$. The situation is similar for the bracket terms, because even if both $x_i$ and $x_j$ are in $\hf$, the bracket will also be in $\hf$ as $\hf$ is a subalgebra, thus also in this case the corresponding term vanishes.  
\end{proof}

The next step in the construction of the spectral sequence  is the computation of the term $E_0$ which is by definition the associated graded space to the filtration, i.e.
$$E_0^{p,q}:={\mathcal F}^pCS^{p+q}(\gf,M)\,/\,{\mathcal F}^{p+1}CS^{p+q}(\gf,M).$$

\begin{lem}
The vector space $E_0^{p,q}$ is isomorphic to
$$E_0^{p,q}\cong\Hom(S^q\hf,\Hom(S^p(\gf/\hf),M)).$$
\end{lem}

\begin{proof}
The isomorphism is induced by the restriction map 
$$r:CS^{p+q}(\gf,M)\to \Hom(S^q\hf,\Hom(S^p(\gf/\hf),M))$$
which restricts the first $q$ arguments of the cochain $f$ to elements of $\hf$. Clearly, the kernel of $r|_{{\mathcal F}^{p}CS^{p+q}(\gf,M)}$ is ${\mathcal F}^{p+1}CS^{p+q}(\gf,M)$. 
It remains to show that $r|_{{\mathcal F}^{p}CS^{p+q}(\gf,M)}$ is surjective. This follows as in the proof of Theorem 1 (p.593) of \cite{HS}: Denote by $\pi:\gf\to\gf/\hf$ the natural quotient map and by $p:\gf\to\hf$ a linear projection onto the subspace $\hf\hookrightarrow \gf$. For a given $\tilde{f}\in \Hom(S^q\hf,\Hom(S^p(\gf/\hf),M))$, we define a preimage $f\in{\mathcal F}^{p}CS^{p+q}(\gf,M)$ of $\tilde{f}$ by 
$$f(x_1,\ldots,x_n):=\sum_{\sigma\in Sh(p,q)}\tilde{f}(p(x_{\sigma(1)}),\ldots,p(x_{\sigma(q)}))
(\pi(x_{\sigma(q+1)}),\ldots,\pi(x_{\sigma(p+q)})),$$
where the sum extends over all $(p,q)$-shuffles. Observe that only one term of the sum 
contributes to $r(f)$, because, as the shuffles keep the order in the first $q$ arguments, only one term has all its elements in $\hf$ and for the other terms, there is $\pi$ applied to at least one element of $\hf$ which gives zero. Thus  $r(f)=\tilde{f}$ as had to be shown. 
\end{proof}

Now we identify the differential $d_1$ which is by definition induced from the Chevalley-Eilenberg coboundary operator (\ref{CE_coboundary}) on $E_0$. 
Observe for this that the quotient $\gf/\hf$ (which is in general not a Lie algebra !) becomes 
naturally an $\hf$-module using the adjoint action, because $\hf$ is a subalgebra. 

\begin{lem}
The differential $d_0$ on $E_0$ identifies to the Chevalley-Eilenberg coboundary operator on 
$CS^\bullet(\hf,\Hom(S^p(\gf/\hf),M))$, the commutative cohomology of the commutative Lie algebra $\hf$ with values in the $\hf$-module $\Hom(S^p(\gf/\hf),M)$. 
\end{lem}

\begin{proof}
The differential 
$$d_0:{\mathcal F}^pCS^{p+q}(\gf,M)\,/\,{\mathcal F}^{p+1}CS^{p+q}(\gf,M) \to {\mathcal F}^pCS^{p+q+1}(\gf,M)\,/\,{\mathcal F}^{p+1}CS^{p+q+1}(\gf,M)$$
is by definition induced by the Chevalley-Eilenberg coboundary operator (\ref{CE_coboundary}). For the identification of $d_0$, we consider again Equation (\ref{CE_coboundary}), but now  with $p+q+1-(p+1)+1=q+1$ elements in $\hf$, because we want to see which terms lie in ${\mathcal F}^{p+1}CS^{p+q+1}(\gf,M)$, and with a cochain $f$
which vanishes on $p+q-p+1=q+1$ elements of $\hf$.  

The action terms do not vanish exactly if $x_i$ is extracted from the first $q$ elements (under the identification of $E_0$ with $\Hom(S^q\hf,\Hom(S^p(\gf/\hf),M))$), otherwise $q+1$ elements in $\hf$ stay as arguments of $f$ and the resulting term is zero. Similarly for the bracket terms, they do not vanish exactly in case $x_i$ and $x_j$ are both in $\hf$, because the number of elements in $\hf$ is then reduced by one. This shows the claim.  
\end{proof}


\section{The spectral sequence associated to an ideal}

One can go further in the identification of the terms in the spectral sequence in case $\hf$ is not only a subalgebra, but an ideal in $\gf$. Suppose now we have a short exact sequence of
commutative Lie algebras
$$0\to \hf\to \gf\to \qf\to 0.$$ 

We have seen in the previous section that 
$$E_1^{p,q}\cong\Hom(S^p(\gf/\hf),HS^q(\hf,M))\cong\Hom(S^p\qf,HS^q(\hf,M)),$$
where we have exchanged arguments in order to express the cohomology with respect to 
$\hf$. 

\begin{lem}
The differential $d_1$ on $E_1$ identifies with the Chevalley-Eilenberg differential of $\qf$ with values in $HS^q(\hf,M)$. 
\end{lem}

\begin{proof}
Note that the action of $\qf$ on $HS^\bullet(\hf,M)$ is well-defined, because the Lie algebra 
$\hf$ acts trivially on its cohomology by the Cartan relation, see Proposition \ref{prop_Cartan_relation}.

It is clear (noting also that $\hf$ acts trivially on $\qf$) that the differential $d_1$ is given by the remaining terms of the Chevalley-Eilenberg coboundary operator (\ref{CE_coboundary}), thus the claim. 
\end{proof}

\begin{cor} \label{HS_for_ideal}
The $E_2$-term of the spectral sequence associated to an ideal reads
$$E_2^{p,q}=HS^p(\qf,HS^q(\hf,M)).$$
\end{cor}


\section{Applications of the spectral sequence} \label{section_applications_HS}

Example D in \cite{FW} is the $2$-dimensional commutative Lie algebra ${\mathfrak N}:=\F e\oplus \F f$ generated by $e$ and $f$ with the relations $[f,f]=e$. The subalgebra $\hf=\langle e\rangle_\F$ is an ideal. Example E in \cite{FW} is the $2$-dimensional Lie algebra 
${\mathfrak a}:=\F h\oplus\F e$ with the bracket determined by $[h,e]=-[e,h]=e$ which becomes commutative over $\F$ of characteristic $2$. The subalgebra $\hf=\langle e\rangle_\F$ is an ideal in ${\mathfrak a}$. Observe that in these examples, the bracket on the subalgebra is zero, thus the coboundary operator on the complex $CS^\bullet(\hf,\F)$ for cohomology with trivial coefficients is zero. 
Thus for these examples and trivial coefficients, it will be not simply an easy argument of cohomology vanishing which computes the Hochschild-Serre spectral sequence, because $HS^q(\hf,\F)=S^q(\hf,\F)$ is non-zero for all $q\geq 0$. Note however that all the spaces 
$S^q(\hf,\F)$ are $1$-dimensional as $\hf$ is $1$-dimensional. 
We then have $HS^p({\mathfrak N}/\hf,HS^q(\hf,\F))=HS^p({\mathfrak N}/\hf,S^q(\hf,\F))$ and the coboundary operator computing the latter cohomology is still zero, because ${\mathfrak N}/\hf=\langle f\rangle_\F$ is a trivial quotient Lie algebra and the action is also zero. Thus  
$$E_2^{p,q}=S^p({\mathfrak N}/\hf,\F)\otimes S^q(\hf,\F)$$
is $1$-dimensional in this case, and the computation of the spectral sequence has to continue with the computation of 
$$d_2^{p,q}:S^p({\mathfrak N}/\hf,\F)\otimes S^q(\hf,\F)\to S^{p+2}({\mathfrak N}/\hf,\F)\otimes S^{q-1}(\hf,\F),$$
which is induced by the bracket $[f,f]=e$. For the latter coboundary operator, it matters only whether there is an even or an odd number of bracket terms. But for $n=p+q$, the coboundary operator has $\frac{n(n+1)}{2}$ bracket terms, which gives zero bracket terms for $n=0$, one bracket term for $n=1$, three bracket terms for $n=2$, six terms for $n=4$ and so on. It follows that $d_2^{p,q}=0$ if and only if $\frac{(p+q)(p+q+1)}{2}$ is even, and $d_2^{p,q}$ is injective otherwise. But $\frac{(p+q)(p+q+1)}{2}$ is even if and only if $4\,|\,(p+q)(p+q+1)$. As only one of the numbers $p+q$ or $p+q+1$ can be even, there are thus two cases: Either $4\,|\,(p+q)$ or $4\,|\,(p+q+1)$. In the first case, $2\,|\,\frac{(p+q)(p+q-1)}{2}$, thus furthermore $d_2^{p-2,q+1}=0$. In the second case $d_2^{p-2,q+1}$ is injective. This permits to compute 
$$E_3^{p,q}=\left\{\begin{array}{ccc} S^p({\mathfrak N}/\hf,\F)\otimes S^q(\hf,\F) & {\rm if} &
4\,|\,(p+q) \\  \{0\} & {\rm if} &
4\,|\,(p+q+1) \\ \{0\} & {\rm else} & \end{array}\right. .$$ 
Observe that in the case $4\,|\,(p+q+1)$, it is the quotient space 
$$(S^p({\mathfrak N}/\hf,\F)\otimes S^q(\hf,\F))\,/\,(S^{p-2}({\mathfrak N}/\hf,\F)\otimes S^{q+1}(\hf,\F))$$
which is zero. 
The spectral sequences has vanishing higher differentials $d_r$ for $r\geq 3$ and the $E_3$-term thus computes the commutative cohomology of ${\mathfrak N}$ with trivial coefficients.
We obtain
$$HS^n({\mathfrak N},\F))=\left\{\begin{array}{ccc} \F^{n+1} & {\rm if} & 4\,|\,n \\
\{0\} & {\rm if} & {\rm not} \end{array}\right. .$$ 

For the Lie algebra ${\mathfrak a}$ and trivial coefficients, the situation is similar: $HS^q(\hf,\F)=S^q(\hf,\F)$ is non-zero for all $q\geq 0$. But here it is already the coboundary operator computing the cohomology $HS^p({\mathfrak a}/\hf,S^q(\hf,\F))$ which involves the original bracket $[h,e]=-[e,h]=e$ in its action terms. As all these action terms are non-zero, it again depends only on the parity of the number of terms whether the coboundary operator is zero or injective. We obtain as before 
$$E_2^{p,q}=\left\{\begin{array}{ccc} S^p({\mathfrak a}/\hf,\F)\otimes S^q(\hf,\F)) & {\rm if} &
4\,|\,(p+q) \\  \{0\} & {\rm if} &
4\,|\,(p+q+1) \\ \{0\} & {\rm else} & \end{array}\right. .$$ 
Again, all higher differentials $d_r$ for $r\geq 2$ are zero and the $E_2$-term computes the commutative cohomology of ${\mathfrak a}$ with trivial values. We obtain again
$$HS^n({\mathfrak a},\F))=\left\{\begin{array}{ccc} \F^{n+1} & {\rm if} & 4\,|\,n \\
\{0\} & {\rm if} & {\rm not} \end{array}\right. .$$

In order to have an easier example, let us consider the following. Let $\hf=\F$ be the $1$-dimensional abelian Lie algebra generated by $x\in\hf$. Let us take as a module the $1$-dimensional $\hf$-module 
$\F_{\lambda}$ generated by $v\in\F_\lambda$ such that the action is given by $x\cdot v=\lambda v$ for a fixed $\lambda\in\F$. Then all symmetric or Leibniz cohomology of $\hf$ 
with values in (the symmetric $\hf$-bimodule) $\F_1$ is zero. Indeed, the complexes are $1$-dimensional in every degree with a coboundary operator having zero bracket part, but the action part is the identity in even degree and zero in odd degree. This gives zero cohomology.  

Therefore, we obtain the following theorem.

\begin{thm} \label{cohomology_vanishing}
Let $\gf$ be commutative Lie algebra which has a $1$-dimensional ideal $\hf$. Suppose that 
$\hf$ acts non-trivially on the module $\F_1$. Then 
$$HS^\bullet(\gf,\F_1)=0.$$ 
\end{thm}

\begin{proof}
By the discussion before the statement of the theorem, we have $HS^\bullet(\hf,\F_1)=0.$
Inserting this in the formula of Corollary \ref{HS_for_ideal} shows the theorem.
\end{proof} 

Using again the formula of Corollary \ref{HS_for_ideal}, one can formulate conditions on a commutative Lie algebra $\gf$ with an ideal $\hf$ such that the $1$-dimensional quotient 
$\gf/\hf$ acts 
non-trivially on $HS^q(\hf,M)\cong\F_1$. In this case, we have also cohomology vanishing.


\section{Comparison to Lie- and Leibniz cohomology}

\subsection{Leibniz algebras and bimodules} 

As already observed, a commutative Lie algebra is in particular a left- and right Leibniz algebra. 

A {\em left Leibniz algebra\/} is an algebra $\lf$ such that every left bracket operator
$L_x:\lf\to\lf$, $y\mapsto [x,y]$ is a derivation. This is equivalent to the identity
\begin{equation}\label{LLI}
[x,[y,z]]=[[x,y],z]+[y,[x,z]]
\end{equation}
for all $x,y,z\in\lf$. There is a similar
definition of a {\em right Leibniz algebra\/}. Leibniz algebras have been studied by Loday and Pirashvili, see \cite{L}, \cite{LP} and \cite{P}. 

Every left Leibniz algebra has an important ideal, its Leibniz kernel, that measures how much
the Leibniz algebra deviates from being a Lie algebra. Namely, let $\lf$ be a left Leibniz algebra
over $\F$. Then $$\leib(\lf):=\langle [x,x]\mid x\in\lf\rangle_\mathbb{F}$$ is called the
{\em Leibniz kernel\/} of $\lf$. The Leibniz kernel $\leib(\lf)$ is an abelian ideal of $\lf$.
By definition of the Leibniz kernel, $\lf_\lie:=\lf/\leib(\lf)$ is a Lie algebra which we call the
{\em canonical Lie algebra\/} associated to $\lf$. In fact, the Leibniz kernel is the smallest
ideal such that the corresponding factor algebra is a Lie algebra. As alternativity 
(i.e. $[x,x]=0$ for all $x\in\gf$) and anticommutativity ($[x,y]=-[y,x]$ for all $x,y\in\gf$) are 
not equivalent in characteristic two, we see that commutative Lie algebras are a class lying 
properly in between Leibniz algebras and Lie algebras:
$${\tt Leib}\supset{\tt CommLie}\supset{\tt Lie}.$$ 
Observe that in this notation, commutative algebras are not a special case of Lie algebras !

Next, we will briefly discuss bimodules of left Leibniz algebras. Let $\lf$ be a
left Leibniz algebra over a field $\F$.
An {\em $\lf$-bimodule\/} is an $\F$-vector space $M$ with $\F$-bilinear left and right $\lf$-operations  $\lf\times M\to M$, $(x,m)\mapsto x\cdot m$ and $M\times
\lf\to M$, $(m,x)\mapsto m\cdot x$ such that
\begin{enumerate}
\item[(LLM)] $[x,y]\cdot m=x\cdot(y\cdot m)+y\cdot(x\cdot m)$,
\item[(LML)] $x\cdot(m\cdot y)=(x\cdot m)\cdot y+m\cdot[x,y]$,
\item[(MLL)] $m\cdot[x,y]=(m\cdot x)\cdot y+x\cdot(m\cdot y)$.
\end{enumerate}
are satisfied for every $m\in M$ and all $x,y\in\lf$. Observe that all this makes sense over the field $\F$ of characteristic $2$ and gives a notion of bimodules for commutative Lie algebras which goes beyond the usual notion of modules. 

\subsection{The three relative complexes} 

In \cite{P} (see also Section 2 of \cite{FW}), Pirashvili constructs a comparison spectral sequence between Chevalley-Eilenberg cohomology and Leibniz cohomology of a Lie algebra. 
Here, we can use the same construction to have three comparison spectral sequences:
There is a comparison spectral sequence 
between Chevalley-Eilenberg cohomology and commutative cohomology for a Lie algebra (passage from $\Lambda^\bullet$ to $S^\bullet$), then a comparison spectral sequence from commutative cohomology to Leibniz cohomology (passage from $S^\bullet$ to $\otimes^\bullet$) and also the usual Pirashvili comparison spectral sequence from Chevalley-Eilenberg cohomology to Leibniz cohomology (passage from $\Lambda^\bullet$ to $\otimes^\bullet$). Observe that all cohomologies involving the Chevalley-Eilenberg complex make only sense for {\it Lie algebras}, while the comparison spectral sequence from $S^\bullet$ to $\otimes^\bullet$ makes sense for the broader class of {\it commutative Lie algebras}. 

Let us first of all describe the different complexes which lead to these comparison spectral sequences. In characteristic $2$, there are two ways of expressing the idea of alternating cochains which are not equivalent (corresponding to {\it alternativity} and {\it antisymmetry} as above mentioned). The first one is:
\begin{equation}    \label{alternating}
\forall x\in\gf:\,\,\,f(\ldots,x,\ldots,x,\ldots)=0.
\end{equation}
This condition is equivalent to saying that $f(x_1,\ldots,x_n)$ vanishes in case the inserted elements $x_1,\ldots,x_n$ of $\gf$ are linearly dependent and describes cochains as maps on the usual exterior algebra $\Lambda^n\gf$. These cochains have the property that if $\gf$ is finite-dimensional of dimension $n$ and $f$ has $n+1$ arguments, then $f$ is zero. We will denote the quotient space of the vector space $\otimes^n\gf$ with respect to the subspace 
$$I_n:=\langle x_1\otimes\ldots\otimes x\otimes\ldots\otimes x\otimes\ldots\otimes x_{n-2}\,|\,
x,x_1,\ldots,x_{n-2}\in\gf\rangle_\F$$ 
by $\Lambda^n\gf:=\otimes^n\gf/I_n$ and the corresponding Chevalley-Eilenberg cochain spaces by 
$C^n(\gf,M):=\Hom_\F(\Lambda^n\gf,M)$. 

On the other hand, the condition 
\begin{equation}     \label{symmetric}
\forall x_1,x_2\in\gf:\,\,\,f(\ldots,x_1,x_2,\ldots)=f(\ldots,x_2,x_1,\ldots)
\end{equation}
describes cochains on the usual symmetric algebra and such a cochain does not have the vanishing property (\ref{alternating}) in general. Observe that condition (\ref{alternating}) implies condition (\ref{symmetric}) by replacing $x=x_1+x_2$. Here the $n$-th graded component $S^n\gf$ of the symmetric algebra $S^\bullet\gf$ is defined as the quotient of $\otimes^n \gf$ by the subspace 
$$J_n:=\langle x_1\otimes\ldots\otimes x\otimes y\otimes\ldots\otimes x_{n-2}+x_1\otimes\ldots\otimes y\otimes x\otimes\ldots\otimes x_{n-2}\,|\,x,y,x_1,\ldots,x_{n-2}\in\gf\rangle_\F$$ 
(and not as the invariants with respect to the symmetric group). We had already introduced 
$CS^n(\gf,M)=\Hom_\F(S^n\gf,M)$. 
The fact that condition (\ref{alternating}) implies condition (\ref{symmetric}) can also be understood as the inclusion $J_n\subset I_n$. 

We thus have three inclusions of subcomplexes (in case $\gf$ is a Lie algebra and the bimodule $M$ is symmetric)
$$i_1:C^\bullet(\gf,M)\hookrightarrow CL^\bullet(\gf,M),$$
$$i_2:C^\bullet(\gf,M)\hookrightarrow CS^\bullet(\gf,M)$$
because condition (\ref{alternating}) implies condition (\ref{symmetric}), and also 
$$i_3:CS^\bullet(\gf,M)\hookrightarrow CL^\bullet(\gf,M)$$
which works for general commutative Lie algebras $\gf$. 

All three give rise to cokernel complexes 
$$C_{\rm rel,\Lambda}^\bullet(\gf,M):=\coker(i_1)[-2],$$ 
$$C_{\rm rel,\Lambda,S}^\bullet(\gf,M):=\coker(i_2)[-2],$$
and 
$$C_{\rm rel,S}^\bullet(\gf,M):=\coker(i_3)[-2]$$ 
respectively (up to a degree shift), and we have the corresponding long exact sequences (like in Proposition 2.2 in \cite{FW}) induced by the short exact sequences of complexes. In each case, the relative cohomology (i.e. the cohomology of the quotient complex $C_{\rm rel,\Lambda}^\bullet(\gf,M)$, $C_{\rm rel,S}^\bullet(\gf,M)$ or $C_{\rm rel,\Lambda,S}^\bullet(\gf,M)$) measures the discrepancy between $H^\bullet(\gf,M)$ and $HL^\bullet(\gf,M)$,
resp. $HS^\bullet(\gf,M)$ and $HL^\bullet(\gf,M)$, resp. $H^\bullet(\gf,M)$ and $HS^\bullet(\gf,M)$ in the sense that if all relative cohomology vanishes, then the two cohomologies coincide.  

\subsection{Comparison Lie- to Leibniz cohomology} 

In this subsection, $\gf$ is a Lie algebra (i.e. $[x,x]=0$ for all $x\in\gf$). 
In order to obtain the comparison spectral sequences, we now introduce filtrations in these three complexes $C_{\rm rel,\Lambda}^\bullet(\gf,M)$,
$C_{\rm rel,S}^\bullet(\gf,M)$ and $C_{\rm rel,\Lambda,S}^\bullet(\gf,M)$ according to the condition of being alternating or symmetric (in the first $p$ arguments). 

Observe that due to the degree shift, a representative of a class in $C_{\rel,\Lambda}^n(\gf,M)$,  $C_{\rel,S}^n(\gf,M)$ or $C_{\rel,\Lambda,S}^n(\gf,M)$ has $n+2$ arguments.
On the quotient complex $C_{\rel,\Lambda}^\bullet(\gf,M)$, there is the following filtration
$${\mathcal F}^pC_{\rel,\Lambda}^{n}(\gf,M)=\{[c]\in C_{\rel,\Lambda}^{n}(\gf,M)\mid c(x_1,\ldots,
x_{n+2})=0\mbox{ if }\exists\,j\leq p+1:\,x_{j-1}=x_j\}\,.$$ 
Note that the condition is independent of the representative $c$ of the class $[c]$. 
This defines a finite decreasing
filtration $${\mathcal F}^0C_{\rel,\Lambda}^{n}(\gf,M)=C_{\rel,\Lambda}^{n}(\gf,M)\supset{\mathcal
F}^1C_{\rel,\Lambda}^{n}(\gf,M)\supset\cdots\supset{\mathcal F}^{n+1}C_{\rel,\Lambda}^{n}(\gf,M)=\{0\}\,,$$ 
whose associated spectral sequence converges thus in the strong (i.e., finite) sense to $H_{\rel,\Lambda}^n(\gf,M)$. 

Like in Section 2 of \cite{FW} (due to Pirashvili \cite{P} !), we have a product map
$$m_1:\Lambda^n\gf\otimes \gf\to \Lambda^{n+1}\gf,\,\,\,x_1\wedge\ldots\wedge x_n\otimes x\mapsto x_1\wedge\ldots\wedge x_n\wedge x.$$
The map $m_1$ induces a monomorphism
$$m^\bullet_1:\overline{C}^\bullet(\gf,\F)\hookrightarrow C^\bullet(\gf,\gf^*),$$
where $\overline{C}^\bullet(\gf,\F)$ is the truncated cochain complex
$$\overline{C}^0(\gf,\F):=0\,\mbox{ and }\,\overline{C}^n
(\gf,\F):=C^n(\gf,\F)\,\,\,\mbox{ for every integer}\,\,\,n\rangle0\,.$$
Accordingly, we have a cochain complex 
$$CR_{\Lambda}^\bullet(\gf):=\coker(m^\bullet_1)[-1],$$
the cokernel of $m^\bullet_1$ (up to a degree shift). We then have a long exact sequence in cohomology from the short exact sequence of complexes like in Proposition 2.1 in \cite{FW}. 

For this filtration/spectral sequence associated to $C_{\rel,\Lambda}^\bullet(\gf,M)$, the arguments of Section 2 in \cite{FW} go through word by word in order to show the following theorem (this is Theorem A of Pirashvili \cite{P}):

\begin{thm}\label{theorem_A}
Let $\gf$ be a Lie algebra, and let $M$ be a left $\gf$-module considered as a symmetric Leibniz
$\gf$-bimodule $M_s$. Then there is a spectral sequence converging to $H^\bullet_{\rel,\Lambda}(\gf,M)$ with second term 
$$E_2^{p,q}=HR^p_{\Lambda}(\gf)\otimes HL^{q}(\gf,M_s)\,.$$
\end{thm}

\subsection{Comparison Lie- to commutative cohomology} 

In this subsection, $\gf$ is still a Lie algebra (i.e. $[x,x]=0$ for all $x\in\gf$). 
The next step is to alter the arguments in Section 2 in \cite{FW} to apply to the filtration/spectral sequence associated to $C_{\rel,\Lambda,S}^\bullet(\gf,M)$. 
We introduce a filtration as follows. In addition to
$$I_n=\langle x_1\otimes\ldots\otimes x_i\otimes\ldots\otimes x_j\otimes\ldots\otimes x_n\,|\,x_i=x_j\,\,\forall x_1,\ldots,x_n\in\gf, \forall 1\leq i<j\leq n\rangle_\F,$$
we introduce
$$I_{n,p}:=\langle x_1\otimes\ldots\otimes x_i\otimes\ldots\otimes x_j\otimes\ldots\otimes x_n\,|\,x_i=x_j\,\,\forall x_1,\ldots,x_n\in\gf, \forall 1\leq i<j\leq p\rangle_\F.$$
We have then 
$$J_n\subset(J_n+I_{n,p}),$$
and denoting
$$\Lambda^p\gf\vee S^{n-p}\gf:=\otimes^n\gf\,/\,(I_{n,p}+J_n),$$
we have 
$${\mathcal F}^pCS^n(\gf,M):=\Hom_\F(\Lambda^p\gf\vee S^{n-p}\gf,M)\subset CS^n(\gf,M)=\Hom_\F(S^n\gf,M).$$
The successive inclusions $(I_{n,p}+J_n)\subset(I_{n,p+1}+J_n)$ lead to a finite decreasing filtration
 $${\mathcal F}^1C_{\rel,\Lambda,S}^{n}(\gf,M)=C_{\rel,\Lambda,S}^{n}(\gf,M)\supset{\mathcal
F}^2C_{\rel,\Lambda}^{n}(\gf,M)\supset\cdots\supset{\mathcal F}^{n}C_{\rel,\Lambda,S}^{n}(\gf,M)=\{0\}\,,$$ 
whose associated spectral sequence converges thus in the strong (i.e., finite) sense to $H_{\rel,\Lambda,S}^n(\gf,M)$.

Again like in Section 2 of \cite{FW}, we have a product map
$$m_2:\Lambda^n\gf\vee \gf\to \Lambda^{n+1}\gf,\,\,\,x_1\wedge\ldots\wedge x_n\vee x\mapsto x_1\wedge\ldots\wedge x_n\wedge x.$$
Here $\Lambda^n\gf\vee \gf\subset S^{n+1}\gf$ is the quotient space of $\otimes^{n+1}\gf$ by the sum $I_{n+1,n}+J_{n+1}$. The map $m_2$ induces a monomorphism
$$m^\bullet_2:\overline{C}^\bullet(\gf,\F)\hookrightarrow C_{\Lambda,S}^\bullet(\gf,\gf^*).$$
Thus we have a cochain complex
$$CR_{\Lambda,S}^\bullet(\gf):=\coker(m^\bullet_2)[-1],$$
the cokernel of $m_2$ (up to a degree shift). We then have a long exact sequence in cohomology from the short exact sequence of complexes like in Proposition 2.1 in \cite{FW}. 

For this filtration/spectral sequence associated to $C_{\rel,\Lambda,S}^\bullet(\gf,M)$, the arguments of Section 2 in \cite{FW} still go through in order to show the following theorem:

\begin{thm}   \label{theorem_A1}
Let $\gf$ be a Lie algebra, and let $M$ be a left $\gf$-module considered as a symmetric Leibniz
$\gf$-bimodule $M_s$. Then there is a spectral sequence converging to $H^\bullet_{\rel,\Lambda,S}(\gf,M)$ with second term 
$$E_2^{p,q}=HR^p_{\Lambda,S}(\gf)\otimes HS^{q}(\gf,M)\,.$$
\end{thm}

\subsection{Comparison commutative- to Leibniz cohomology} 

In this subsection, $\gf$ is now an arbitrary commutative Lie algebra (i.e. we have only 
$[x,y]=[y,x]$ for all $x,y\in\gf$). 
For the comparison between commutative cohomology and Leibniz cohomology, we reason 
as follows. The version of $J_n$ which concerns only the first $p$ tensor factors is 
\begin{align*}
J_{n,p}:=\langle x_1\otimes\ldots\otimes x_i\otimes x_{i+1}\otimes\ldots\otimes x_n+x_1\otimes\ldots\otimes x_{i+1}\otimes x_i\otimes\ldots\otimes x_n\,| \\
x_1,\ldots,x_n\in\gf\,\,\,{\rm and}\,\,\,i\leq p-1\rangle_\F.
\end{align*}

The sequence of inclusions
$$\{0\}\subset J_{n,2}\subset J_{n,3}\subset\ldots\subset J_{n,n}=J_n$$
induces a sequence of surjections
$$\otimes^n\gf\twoheadrightarrow S^2\gf\otimes \otimes^{n-2}\gf\twoheadrightarrow\ldots\twoheadrightarrow S^n\gf,$$
where we identify $S^p\gf\otimes \otimes^{n-p}\gf=\otimes^n\gf\,/\,J_{n,p}$, which induces  
in turn a sequence of monomorphisms
 $${\mathcal F}^nCL^{n}(\gf,M)=CS^{n}(\gf,M)\subset{\mathcal
F}^{n-1}CL^{n}(\gf,M)\subset\cdots\subset{\mathcal F}^{1}CL^{n}(\gf,M)=CL^n(\gf,M),$$
where by definition ${\mathcal F}^{p}CL^{n}(\gf,M):=\Hom_\F(S^p\gf\otimes \otimes^{n-p}\gf,M)$. We interpret this sequence of monomorphisms as a finite decreasing filtration of the quotient complex
$${\mathcal F}^1C_{\rel,S}^{n}(\gf,M)=C_{\rel,S}^{n}(\gf,M)\supset{\mathcal
F}^2C_{\rel,S}^{n}(\gf,M)\supset\cdots\supset{\mathcal F}^{n}C_{\rel,S}^{n}(\gf,M)=\{0\},$$
which arises from dividing the above sequence of monomorphisms by $CS^n(\gf,M)$.  
We need to show that this filtration is compatible with the Chevalley-Eilenberg differential,
i.e. 
$$d{\mathcal F}^{p}C_{\rel,S}^{n}(\gf,M)\subset{\mathcal F}^{p}C_{\rel,S}^{n+1}(\gf,M).$$ 
Indeed, given a cochain $f\in\Hom_\F(S^p\gf\otimes \otimes^{n-p}\gf,M)$, representing a class in ${\mathcal F}^{p}C_{\rel,S}^{n}(\gf,M)$, we have to show that $df$, given by Equation (\ref{CE_coboundary}), is still symmetric in the first $p$ entries. But this is clear from the symmetry of the bracket of $\gf$ together with the symmetry of the terms of (\ref{CE_coboundary}). In conclusion, we have a spectral sequence converging finitely to  $H^\bullet_{\rel,S}(\gf,M)$. 

Again like in Section 2 of \cite{FW}, we have a product map
$$m_3:S^n\gf\otimes \gf\to S^{n+1}\gf,\,\,\,x_1\vee\ldots\vee x_n\otimes x\mapsto x_1\vee\ldots\vee x_n\vee x,$$

The map $m_3$ induces a monomorphism
$$m^\bullet_3:\overline{CS}^\bullet(\gf,\F)\hookrightarrow CS^\bullet(\gf,\gf^*).$$
Accordingly, we have a cochain complex 
$$CR_{S}^\bullet(\gf):=\coker(m^\bullet_3)[-1]$$
the cokernel of $m^\bullet_3$ (up to a degree shift). We then have a long exact sequence in cohomology from the short exact sequence of complexes like in Proposition 2.1 in \cite{FW}. 

Now the reasoning of Section 2 in \cite{FW} goes through (mutatis mutandis) in order to show the following theorem. 

\begin{thm}  \label{theorem_A2}
Let $\gf$ be a commutative Lie algebra, and let $M$ be a left $\gf$-module considered as a symmetric Leibniz
$\gf$-bimodule $M_s$. Then there is a spectral sequence converging to $H^\bullet_{\rel,S}(\gf,M)$ with second term 
$$E_2^{p,q}=HR^p_{S}(\gf)\otimes HL^{q}(\gf,M)\,.$$
\end{thm}


\section{Applications of the comparison spectral sequences}

By general homological algebra, the spectral sequences of Theorems  \ref{theorem_A}, \ref{theorem_A1}, \ref{theorem_A2} imply (using the long exact sequences arising from the definitions of $H^\bullet_{\rel,\Lambda}(\gf,M)$, $H^\bullet_{\rel,\Lambda,S}(\gf,M)$ and $H^\bullet_{\rel,S}(\gf,M)$ resp.) lead (exactly as in the proof of Theorem 2.6 in \cite{FW}) to the following three theorems. 

\begin{thm}\label{vanlie}
Let $\gf$ be a Lie algebra, let $M$ be a left $\gf$-module considered as a symmetric Leibniz
$\gf$-bimodule $M_s$, and let $n$ be a non-negative integer. If $H^k(\gf,M)=0$ for every
integer $k$ with $0\le k\le n$, then $HL^k(\gf,M_s)=0$ for every integer $k$ with $0\le k\le
n$ and $HL^{n+1}(\gf,M_s)\cong H^{n+1}(\gf,M)$ as well as $HL^{n+2}(\gf,M_s)\cong
H^{n+2}(\gf,M)$. In particular, $H^\bullet(\gf,M)=0$ implies that $HL^\bullet(\gf,M_s)
=0$.
\end{thm}

This is exactly Theorem 2.6 in \cite{FW}. It has a converse stated in the Remark after its proof in \cite{FW}. 

\begin{thm}\label{vanlie1}
Let $\gf$ be a Lie algebra, let $M$ be a left $\gf$-module considered as a symmetric Leibniz
$\gf$-bimodule $M_s$, and let $n$ be a non-negative integer. If $H^k(\gf,M)=0$ for every
integer $k$ with $0\le k\le n$, then $HS^k(\gf,M_s)=0$ for every integer $k$ with $0\le k\le
n$ and $HS^{n+1}(\gf,M_s)\cong H^{n+1}(\gf,M)$ as well as $HS^{n+2}(\gf,M_s)\cong
H^{n+2}(\gf,M)$. In particular, $H^\bullet(\gf,M)=0$ implies that $HS^\bullet(\gf,M_s)
=0$.
\end{thm}

And for the more general case of commutative Lie algebras:

\begin{thm}\label{vanlie2}
Let $\gf$ be a commutative Lie algebra, let $M$ be a left $\gf$-module considered as a symmetric Leibniz
$\gf$-bimodule $M_s$, and let $n$ be a non-negative integer. If $HS^k(\gf,M)=0$ for every
integer $k$ with $0\le k\le n$, then $HL^k(\gf,M_s)=0$ for every integer $k$ with $0\le k\le
n$ and $HL^{n+1}(\gf,M_s)\cong HS^{n+1}(\gf,M)$ as well as $HL^{n+2}(\gf,M_s)\cong
HS^{n+2}(\gf,M)$. In particular, $HS^\bullet(\gf,M)=0$ implies that $HL^\bullet(\gf,M_s)
=0$.
\end{thm}

The analogous converse statements (like in the Remark after Theorem 2.6 in \cite{FW}) are 
obviously also true. Indeed, suppose for example that $HL^\bullet(\gf,M_s)=0$. Then by the long exact sequence $H^{p-2}_{\rel,S}(\gf,M)=HS^p(\gf,M)$  for all $p\geq 2$ (while $HS^0(\gf,M)=HS^1(\gf,M)=0$). On the other hand, the spectral sequence in Theorem  \ref{theorem_A2} shows that the $E_2$-term, and thus the relative cohomology $H^{p}_{\rel,S}(\gf,M)$ is zero for all $p\geq 0$. Therefore all cohomology $HS^\bullet(\gf,M)=0$. 

One may also formulate conditions relying on a vanishing of the tensor factor $HR_{\Lambda}^k(\gf,M)$ for $0\leq k\leq n$ in order to obtain from the spectral sequence (together with the long exact sequence linking $H^\bullet(\gf,M)$ and $HL^\bullet(\gf,M)$) the isomorphy between $H^\bullet(\gf,M)$ and $HL^\bullet(\gf,M)$. 
For example, the isomorphy of $H^2(\gf,\F)$ and $HL^1(\gf,\gf)$ on the one hand and $H^3(\gf,\F)$ and $HL^2(\gf,\gf)$ on the other hand imply the vanishing of $HR_{\Lambda}^0(\gf)$ which then can be reported into the spectral sequence of Theorem \ref{theorem_A}
to yield $H_{\rel,\Lambda}^0(\gf,M)=0$ and thus (by the long exact sequence) $H^2(\gf,M)\cong HL^2(\gf,M)$. Obviously, this can be done with a vanishing of the first $n$ spaces  
$HR_{\Lambda}^k(\gf)$. Similar methods can be applied to the situations in Theorems \ref{theorem_A1} and \ref{theorem_A2}, too. We refrain form stating all the consequences as theorems. 

As a concrete computation, we may consider similarly to Section \ref{section_applications_HS}a Lie algebra $\gf$ which has a $1$-dimensional ideal $\hf$, and consider cohomology with values in the $1$-dimensional module $\F_1$. By the above theorems, 
we obtain directly with Theorem \ref{cohomology_vanishing}:

\begin{thm} \label{vanishing_all_cohomologies}
Let $\gf$ be a Lie algebra which has a $1$-dimensional ideal $\hf$. Suppose that 
$\hf$ acts non-trivially on the module $\F_1$. Then 
$$HS^\bullet(\gf,\F_1)=HL^\bullet(\gf,\F_1)=H^\bullet(\gf,\F_1)=0.$$ 
\end{thm}


\vspace{.5cm}
\noindent {\bf Acknowledgments.} 
The author thanks the organisors of the Porto conference on non-associative algebras in 2019 where this work was initiated, for inviting him. He thanks furthermore Pasha Zusmanovich for discussion and for sending him a copy of \cite{LZ}.



\end{document}